\documentclass[12pt,reqno]{amsart}

\usepackage{amsfonts, amsthm, amsmath}
\allowdisplaybreaks[4]

\usepackage{rotating}

\usepackage{tikz}

\usepackage{graphics}

\usepackage{amssymb}

\usepackage{amscd}

\usepackage[latin2]{inputenc}

\usepackage{t1enc}

\usepackage[mathscr]{eucal}

\usepackage{indentfirst}

\usepackage{graphicx}

\usepackage{graphics}

\usepackage{pict2e}

\usepackage{mathrsfs}

\usepackage{enumerate}
\usepackage[pagebackref]{hyperref}
\hypersetup{backref, pagebackref, colorlinks=true}
\usepackage{cite}
\usepackage{color}
\usepackage{epic}
\usepackage{hyperref} 
\numberwithin{equation}{section}
\theoremstyle{plain}

\newtheorem{theorem}{Theorem}[section]

\newtheorem{lemma}[theorem]{Lemma}

\newtheorem{conjecture}[theorem]{Conjecture}

\theoremstyle{definition}

\newtheorem{?}[theorem]{Problem}

\def\boxit#1{\leavevmode\hbox{\vrule\vtop{\vbox{\kern.33333pt\hrule
    \kern1pt\hbox{\kern1pt\vbox{#1}\kern1pt}}\kern1pt\hrule}\vrule}}

\usepackage{collectbox}



\newcommand{\f}[1]{\ifthenelse{\equal{#1}{1}}{(q;q)_\infty}{(q^{#1};q^{#1})_{\infty}}}

\begin{document}
\title[Vanishing coefficients in some $q$-series expansions]{Vanishing coefficients in some $q$-series expansions}

\author[D. Tang]{Dazhao Tang}

\address[Dazhao Tang]{College of Mathematics and Statistics, Chongqing University, Huxi Campus LD206, Chongqing 401331, P.R. China}
\email{dazhaotang@sina.com}

\date{\today}

\begin{abstract}
Motivated by the recent work of Hirschhorn on vanishing coefficients of the arithmetic progressions in certain $q$-series expansions, we study some variants of these $q$-series and prove some comparable results. For instance, let
\begin{align*}
(-q,-q^{4};q^{5})_{\infty}^{2}(q^{4},q^{6};q^{10})_{\infty}=\sum_{n=0}^{\infty}a_{1}(n)q^{n},
\end{align*}
then
\begin{align*}
a_{1}(5n+3)=0.
\end{align*}
\end{abstract}



\maketitle

\section{Introduction}
The study of vanishing coefficients in infinite product expansions, which was derived from the Hardy-Ramanujan-Rademacher expansions for quotients of certain infinite products, can be traced back to Richmond and Szekeres \cite{RS1978}. Soon after, Andrews and Bressoud \cite{AB1979} proved a general theorem for certain Rogers-Ramanujan type infinite products, which contains the results of Richmond and Szekeres as special cases. Later, Alladi and Gordon \cite{AG1994} arrived at a more general result with some restrictions. In a recent paper, among other things,  McLaughlin \cite{Lau2015} further generalized the results of Alladi and Gordon.

Quite recently, Hirschhorn \cite{Hir2018} studied the following two $q$-series:
\begin{align}
(-q,-q^{4};q^{5})_{\infty}(q,q^{9};q^{10})_{\infty}^{3} &=\sum_{n=0}^{\infty}a(n)q^{n},\label{H1}\\
(-q^{2},-q^{3};q^{5})_{\infty}(q^{3},q^{7};q^{10})_{\infty}^{3} &=\sum_{n=0}^{\infty}b(n)q^{n}.\label{H2}
\end{align}
He proved that
\begin{align*}
a(5n+2) &=a(5n+4)=0,\\
b(5n+1) &=b(5n+4)=0.
\end{align*}

Here and in what follows, we adopt the following customary $q$-series notations:
\begin{align*}
(a;q)_{\infty} &:=\prod_{n=0}^{\infty}(1-aq^{n}),\\
(a_{1},a_{2},\cdots,a_{m};q)_{\infty} &:=(a_{1};q)_{\infty}(a_{2};q)_{\infty}\cdots(a_{m};q)_{\infty},\quad\textrm{for}\quad|q|<1.
\end{align*}

In this paper, we consider some variants of \eqref{H1} and \eqref{H2}, and obtain some comparable results on vanishing coefficients in these $q$-series expansions.

Define
\begin{align*}
(-q,-q^{4};q^{5})_{\infty}^{2}(q^{4},q^{6};q^{10})_{\infty} &=\sum_{n=0}^{\infty}a_{1}(n)q^{n},\\
(-q^{2},-q^{3};q^{5})_{\infty}^{2}(q^{2},q^{8};q^{10})_{\infty} &=\sum_{n=0}^{\infty}b_{1}(n)q^{n},\\
(-q,-q^{4};q^{5})_{\infty}^{3}(q^{2},q^{8};q^{10})_{\infty} &=\sum_{n=0}^{\infty}a_{2}(n)q^{n},\\
(-q^{2},-q^{3};q^{5})_{\infty}^{3}(q^{4},q^{6};q^{10})_{\infty} &=\sum_{n=0}^{\infty}b_{2}(n)q^{n},\\
(-q,-q^{4};q^{5})_{\infty}^{3}(q^{3},q^{7};q^{10})_{\infty} &=\sum_{n=0}^{\infty}a_{3}(n)q^{n},\\
(-q^{2},-q^{3};q^{5})_{\infty}^{3}(q,q^{9};q^{10})_{\infty} &=\sum_{n=0}^{\infty}b_{3}(n)q^{n}.
\end{align*}

\begin{theorem}
For any integer $n\geq0$,
\begin{align}
a_{1}(5n+3) &=0,\label{a1}\\
b_{1}(5n+1) &=0,\label{b1}\\
a_{2}(5n+4) &=0,\label{a2}\\
b_{2}(5n+1) &=0,\label{b2}\\
a_{3}(5n+3) &=a_{3}(5n+4)=0,\label{a3}\\
b_{3}(5n+3) &=b_{3}(5n+4)=0.\label{b3}
\end{align}
\end{theorem}

\section{Proofs}
Ramanujan's theta function is defined by
\begin{align*}
f(a,b) &:=\sum_{n=-\infty}^{\infty}a^{n(n+1)/2}b^{n(n-1)/2},
\end{align*}
where $|ab|<1$. The function $f(a,b)$ enjoys the well-known Jacobi triple product identity \cite[p. 35, Entry 19]{Ber1991}:
\begin{align*}
f(a,b) &=(-a,-b,ab;ab)_{\infty}.
\end{align*}

For notational convenience, denote
\begin{align*}
E_{j} &:=(q^{j};q^{j})_{\infty}.
\end{align*}

Let $k, l$ be positive integers and $G(q)=\sum\limits_{n=0}^{\infty}g(n)q^{n}$ be a formal power series. Define an operator $H_{k,l}$ by
\begin{align*}
H_{k,l}\left(G(q)\right) &:=\sum_{n=0}^{\infty}g(kn+l)q^{kn+l}.
\end{align*}

Recall that Ramanujan's classical theta functions $\varphi(q)$ and $\psi(q)$ are given by \cite[Eqs. (1.5.4) and (1.5.5)]{Hirb2017}:
\begin{align*}
\varphi(q) &:=f(q,q)=\sum_{n=-\infty}^{\infty}q^{n^{2}}=\dfrac{E_{2}^{5}}{E_{1}^{2}E_{4}^{2}},\\
\psi(q) &:=f(q,q^{3})=\sum_{n=0}^{\infty}q^{n(n+1)/2}=\dfrac{E_{2}^{2}}{E_{1}}.
\end{align*}

\begin{lemma}\label{lemma1}
Define
\begin{align*}
S_{1} &=\sum_{m,n=-\infty}^{\infty}q^{20m^{2}+2m+20n^{2}+6n},\quad S_{2}=\sum_{m,n=-\infty}^{\infty}q^{20m^{2}+18m+20n^{2}+6n},\\
S_{3} &=\sum_{m,n=-\infty}^{\infty}q^{20m^{2}+2m+20n^{2}+14n},\quad S_{4}=\sum_{m,n=-\infty}^{\infty}q^{20m^{2}+18m+20n^{2}+14n},\\
S_{5} &=\sum_{m,n=-\infty}^{\infty}q^{20m^{2}+2m+20n^{2}+4n},\quad S_{6}=\sum_{m,n=-\infty}^{\infty}q^{20m^{2}+18m+20n^{2}+4n},\\
S_{7} &=\sum_{m,n=-\infty}^{\infty}q^{20m^{2}+2m+20n^{2}+16n},\quad S_{8}=\sum_{m,n=-\infty}^{\infty}q^{20m^{2}+18m+20n^{2}+16n}.
\end{align*}
Then
\begin{align*}
H_{5,3}\left(S_{1}-q^{4}S_{2}\right) &=H_{5,3}\left(q^{2}S_{3}-q^{6}S_{4}\right)=H_{5,3}\left(qS_{5}-q^{5}S_{6}\right)=H_{5,3}\left(q^{4}S_{7}-q^{8}S_{8}\right)=0.
\end{align*}
\end{lemma}

\begin{proof}
We only prove $H_{5,3}\left(S_{1}-q^{4}S_{2}\right)=0$, and the remaining cases are similar.

In $S_{1}$, if $2m+6n\equiv3\pmod{5}$, then $2m+n\equiv-2\pmod{5}$. Of course, $m-2n\equiv-1\pmod{5}$. Assume $2m+n=5r-2$ and $m-2n=-5s-1$, it follows that $m=2r-s-1$ and $n=r+2s$. Thus $H_{5,3}(S_{1})$ equals
\begin{align*}
 &\sum_{r,s=-\infty}^{\infty}q^{20(2r-s-1)^{2}+2(2r-s-1)+20(r+2s)^{2}+6(r+2s)}\\
 &=q^{18}\sum_{r,s=-\infty}^{\infty}q^{100r^{2}+100s^{2}-70r+50s}=q^{18}\sum_{r,s=-\infty}^{\infty}q^{100r^{2}+100s^{2}+70r+50s}.
\end{align*}

In $S_{2}$, $18m+6n\equiv-1$, $-2m+n\equiv-1$, that is, $2m-n\equiv1$, $m+2n\equiv-2$. Let $2m-n=5r+1$ and $m+2n=-5s-2$, then $m=2r-s$, $n=-r-2s-1$. Therefore, $H_{5,3}\left(q^{4}S_{2}\right)$ is
\begin{align*}
 &q^{4}\sum_{r,s=-\infty}^{\infty}q^{20(2r-s)^{2}+18(2r-s)+20(-r-2s-1)^{2}+6(-r-2s-1)}\\
 &=q^{18}\sum_{r,s=-\infty}^{\infty}q^{100r^{2}+100s^{2}+70r+50s},
\end{align*}
as desired.
\end{proof}

The above $q$-series manipulation was developed by Hirschhorn \cite{Hir2018}, therefore we called it \emph{Hirschhorn's operation} in the sequel.

We have
\begin{align*}
(-q,-q^{4};q^{5})_{\infty}^{2}=\dfrac{f(q,q^{4})^{2}}{E_{5}^{2}} &=\dfrac{1}{E_{5}^{2}}\sum_{m,n=-\infty}^{\infty}q^{(5m^{2}+3m+5n^{2}+3n)/2}\\
 &=\dfrac{1}{E_{5}^{2}}\Bigg(\sum_{r,s=-\infty}^{\infty}q^{(5(r+s)^{2}+3(r+s)+5(r-s)^{2}+3(r-s))/2}\\
 &\quad+\sum_{r,s=-\infty}^{\infty}q^{(5(r+s-1)^{2}+3(r+s-1)+5(r-s)^{2}+3(r-s))/2}\Bigg)\\
 &=\dfrac{E_{10}^{5}}{E_{5}^{4}E_{20}^{2}}\sum_{n=-\infty}^{\infty}q^{5n^{2}+3n}+2q\dfrac{E_{20}^{2}}{E_{5}^{2}E_{10}}
 \sum_{n=-\infty}^{\infty}q^{5n^{2}+2n}\\
 &=\dfrac{E_{10}^{5}}{E_{5}^{4}E_{20}^{2}}\Bigg(\sum_{n=-\infty}^{\infty}q^{20n^{2}+6n}+q^{2}\sum_{n=-\infty}^{\infty}q^{20n^{2}+14n}\Bigg)\\
 &\quad+2\dfrac{E_{20}^{2}}{E_{5}^{2}E_{10}}\Bigg(q\sum_{n=-\infty}^{\infty}q^{20n^{2}+4n}+q^{4}\sum_{n=-\infty}^{\infty}q^{20n^{2}+16n}\Bigg).
\end{align*}

Moreover,
\begin{align*}
(q^{4},q^{6};q^{10})_{\infty} &=\dfrac{1}{E_{10}}\sum_{m=-\infty}^{\infty}(-1)^{m}q^{5m^{2}+m}\\
 &=\dfrac{1}{E_{10}}\Bigg(\sum_{m=-\infty}^{\infty}q^{20m^{2}+2m}-q^{4}\sum_{m=-\infty}^{\infty}q^{20m^{2}+18m}\Bigg).
\end{align*}

We then get
\begin{align*}
\sum_{n=0}^{\infty}a_{1}(n)q^{n} &=\dfrac{1}{E_{10}}\Bigg(\sum_{m=-\infty}^{\infty}q^{20m^{2}+2m}-q^{4}\sum_{m=-\infty}^{\infty}q^{20m^{2}+18m}\Bigg)\\
 &\quad\times\Bigg(\dfrac{E_{10}^{5}}{E_{5}^{4}E_{20}^{2}}\sum_{n=-\infty}^{\infty}q^{20n^{2}+6n}+q^{2}\dfrac{E_{10}^{5}}{E_{5}^{4}E_{20}^{2}}
 \sum_{n=-\infty}^{\infty}q^{20n^{2}+14n}\\
 &\quad+2q\dfrac{E_{20}^{2}}{E_{5}^{2}E_{10}}\sum_{n=-\infty}^{\infty}q^{20n^{2}+4n}+2q^{4}\dfrac{E_{20}^{2}}{E_{5}^{2}E_{10}}
 \sum_{n=-\infty}^{\infty}q^{20n^{2}+16n}\Bigg)\\
 &=\dfrac{E_{10}^{4}}{E_{5}^{4}E_{20}^{2}}\big(S_{1}-q^{4}S_{2}+q^{2}S_{3}-q^{6}S_{4}\big)
 +\dfrac{2E_{20}^{2}}{E_{5}^{2}E_{10}^{2}}\big(qS_{5}-q^{5}S_{6}+q^{4}S_{7}-q^{8}S_{8}\big).
\end{align*}

In view of Lemma \ref{lemma1}, we obtain \eqref{a1}.

The proof of \eqref{b1} is similar so is omitted here.

Now, we are ready to prove \eqref{a2}--\eqref{b3}.

We start with
\begin{align*}
\sum_{n=0}^{\infty}a_{2}(n)q^{n} &=\dfrac{f(q,q^{4})}{E_{5}E_{10}}\Bigg(\sum_{m=-\infty}^{\infty}q^{20m^{2}+6m}-q^{2}\sum_{m=-\infty}^{\infty}q^{20m^{2}+14m}\Bigg)\\
 &\quad\times\Bigg(\dfrac{E_{10}^{5}}{E_{5}^{4}E_{20}^{2}}\sum_{n=-\infty}^{\infty}q^{20n^{2}+6n}+q^{2}\dfrac{E_{10}^{5}}{E_{5}^{4}E_{20}^{2}}
 \sum_{n=-\infty}^{\infty}q^{20n^{2}+14n}\\
 &\quad+2q\dfrac{E_{20}^{2}}{E_{5}^{2}E_{10}}\sum_{n=-\infty}^{\infty}q^{20n^{2}+4n}+2q^{4}\dfrac{E_{20}^{2}}{E_{5}^{2}E_{10}}
 \sum_{n=-\infty}^{\infty}q^{20n^{2}+16n}\Bigg)\\
 &=\dfrac{E_{10}^{4}}{E_{5}^{5}E_{20}^{2}}f(q,q^{4})\Bigg(\sum_{r,s=-\infty}^{\infty}q^{20(r+s)^{2}+6(r+s)+20(r-s)^{2}+6(r-s)}\\
 &\quad+\sum_{r,s=-\infty}^{\infty}q^{20(r+s-1)^{2}+6(r+s-1)+20(r-s)^{2}+6(r-s)}\\
 &\quad-q^{4}\sum_{r,s=-\infty}^{\infty}q^{20(r+s)^{2}+14(r+s)+20(r-s)^{2}+14(r-s)}\\
 &\quad-q^{4}\sum_{r,s=-\infty}^{\infty}q^{20(r+s-1)^{2}+14(r+s-1)+20(r-s)^{2}+14(r-s)}\Bigg)\\
 &\quad+\dfrac{2E_{20}^{2}}{E_{5}^{3}E_{10}^{2}}f(q,q^{4})\Bigg(q\sum_{r,s=-\infty}^{\infty}q^{20(r+s)^{2}+6(r+s)+20(r-s)^{2}+4(r-s)}\\
 &\quad+q\sum_{r,s=-\infty}^{\infty}q^{20(r+s-1)^{2}+6(r+s-1)+20(r-s)^{2}+4(r-s)}\\
 &\quad+q^{4}\sum_{r,s=-\infty}^{\infty}q^{20(r+s)^{2}+6(r+s)+20(r-s)^{2}+16(r-s)}\\
 &\quad+q^{4}\sum_{r,s=-\infty}^{\infty}q^{20(r+s-1)^{2}+6(r+s-1)+20(r-s)^{2}+16(r-s)}\\
 &\quad-q^{3}\sum_{r,s=-\infty}^{\infty}q^{20(r+s)^{2}+14(r+s)+20(r-s)^{2}+4(r-s)}\\
 &\quad-q^{3}\sum_{r,s=-\infty}^{\infty}q^{20(r+s-1)^{2}+14(r+s-1)+20(r-s)^{2}+4(r-s)}\\
 &\quad-q^{6}\sum_{r,s=-\infty}^{\infty}q^{20(r+s)^{2}+14(r+s)+20(r-s)^{2}+16(r-s)}\\
 &\quad-q^{6}\sum_{r,s=-\infty}^{\infty}q^{20(r+s-1)^{2}+14(r+s-1)+20(r-s)^{2}+16(r-s)}\Bigg)\\
 &=\dfrac{E_{10}^{4}E_{80}^{5}f(q,q^{4})}{E_{5}^{5}E_{20}^{2}E_{40}^{2}E_{160}^{2}}\Bigg(\sum_{n=-\infty}^{\infty}q^{40n^{2}+12n}
 -q^{4}\sum_{n=-\infty}^{\infty}q^{40n^{2}+28n}\Bigg)\\
 &\quad+\dfrac{2E_{10}^{4}E_{160}^{2}f(q,q^{4})}{E_{5}^{5}E_{20}^{2}E_{80}}\Bigg(q^{14}\sum_{n=-\infty}^{\infty}q^{40n^{2}+28n}-q^{10}
 \sum_{n=-\infty}^{\infty}q^{40n^{2}+12n}\Bigg)\\
 &\quad+\dfrac{2E_{20}^{2}f(q^{30},q^{50})f(q,q^{4})}{E_{5}^{3}E_{10}^{2}}\Bigg(q\sum_{n=-\infty}^{\infty}q^{40n^{2}+2n}
 -q^{10}\sum_{n=-\infty}^{\infty}q^{40n^{2}+38n}\\
 &\quad+q^{4}\sum_{n=-\infty}^{\infty}q^{40n^{2}+22n}-q^{3}\sum_{n=-\infty}^{\infty}q^{40n^{2}+18n}\Bigg)\\
 &\quad+\dfrac{2E_{20}^{2}f(q^{10},q^{70})f(q,q^{4})}{E_{5}^{3}E_{10}^{2}}
 \Bigg(q^{15}\sum_{n=-\infty}^{\infty}q^{40n^{2}+38n}-q^{6}\sum_{n=-\infty}^{\infty}q^{40n^{2}+2n}\\
 &\quad+q^{8}\sum_{n=-\infty}^{\infty}q^{40n^{2}+18n}-q^{9}\sum_{n=-\infty}^{\infty}q^{40n^{2}+22n}\Bigg).
\end{align*}

Denote
\begin{align*}
A_{1} &:=f(q,q^{4})\sum_{n=-\infty}^{\infty}q^{40n^{2}+12n}-q^{4}f(q,q^{4})\sum_{n=-\infty}^{\infty}q^{40n^{2}+28n},\\
A_{2} &:=q^{14}f(q,q^{4})\sum_{n=-\infty}^{\infty}q^{40n^{2}+28n}-q^{10}f(q,q^{4})\sum_{n=-\infty}^{\infty}q^{40n^{2}+12n},\\
A_{3} &:=qf(q,q^{4})\sum_{n=-\infty}^{\infty}q^{40n^{2}+2n}-q^{10}f(q,q^{4})\sum_{n=-\infty}^{\infty}q^{40n^{2}+38n},\\
A_{4} &:=q^{4}f(q,q^{4})\sum_{n=-\infty}^{\infty}q^{40n^{2}+22n}-q^{3}f(q,q^{4})\sum_{n=-\infty}^{\infty}q^{40n^{2}+18n},\\
A_{5} &:=q^{15}f(q,q^{4})\sum_{n=-\infty}^{\infty}q^{40n^{2}+38n}-q^{6}f(q,q^{4})\sum_{n=-\infty}^{\infty}q^{40n^{2}+2n},\\
A_{6} &:=q^{8}f(q,q^{4})\sum_{n=-\infty}^{\infty}q^{40n^{2}+18n}-q^{9}f(q,q^{4})\sum_{n=-\infty}^{\infty}q^{40n^{2}+22n}.
\end{align*}

Next, we prove that
\begin{align*}
H_{5,4}(A_{i})=0,\quad \textrm{for}\quad1\leq i\leq6.
\end{align*}
We only prove the case $A_{1}$ here because the proofs of remaining cases are similar.

Also,
\begin{align*}
f(q,q^{4}) &=\sum_{m=-\infty}^{\infty}q^{(5m^{2}+3m)/2}\\
 &=\sum_{m=-\infty}^{\infty}q^{10m^{2}+3m}+q\sum_{m=-\infty}^{\infty}q^{10m^{2}+7m}\\
 &=\sum_{m=-\infty}^{\infty}q^{40m^{2}+6m}+q^{7}\sum_{m=-\infty}^{\infty}q^{40m^{2}+34m}
 +q\sum_{m=-\infty}^{\infty}q^{40m^{2}+14m}+q^{4}\sum_{m=-\infty}^{\infty}q^{40m^{2}+26m}.
\end{align*}
Therefore,
\begin{align*}
A_{1} &=P_{1}-P_{2}+P_{3}-P_{4}+P_{5}-P_{6}+P_{7}-P_{8},
\end{align*}
where
\begin{align*}
P_{1} &=\sum_{m,n=-\infty}^{\infty}q^{40m^{2}+6m+40n^{2}+12n},\quad P_{2}=q^{8}\sum_{m,n=-\infty}^{\infty}q^{40m^{2}+26m+40n^{2}+28n},\\
P_{3} &=q\sum_{m,n=-\infty}^{\infty}q^{40m^{2}+14m+40n^{2}+12n},\quad P_{4}=q^{11}\sum_{m,n=-\infty}^{\infty}q^{40m^{2}+34m+40n^{2}+28n},\\
P_{5} &=q^{4}\sum_{m,n=-\infty}^{\infty}q^{40m^{2}+26m+40n^{2}+12n},\quad P_{6}=q^{4}\sum_{m,n=-\infty}^{\infty}q^{40m^{2}+6m+40n^{2}+28n},\\
P_{7} &=q^{7}\sum_{m,n=-\infty}^{\infty}q^{40m^{2}+34m+40n^{2}+12n},\quad P_{8}=q^{5}\sum_{m,n=-\infty}^{\infty}q^{40m^{2}+14m+40n^{2}+28n}.
\end{align*}

By Hirschhorn's operation, we have
\begin{align*}
H_{5,4}\left(P_{2i-1}-P_{2i}\right)=0,\quad \textrm{for}\quad1\leq i\leq4.
\end{align*}
This proves \eqref{a2}.

Similarly, we get
\begin{align*}
\sum_{n=0}^{\infty}a_{3}(n)q^{n} &=\dfrac{f(q,q^{4})}{E_{5}^{3}E_{10}}\Bigg(\sum_{m=-\infty}^{\infty}q^{20m^{2}+4m}-q^{3}\sum_{m=-\infty}^{\infty}q^{20m^{2}+16m}\Bigg)\\
 &\quad\times\Bigg(\dfrac{E_{10}^{5}}{E_{5}^{2}E_{20}^{2}}\sum_{n=-\infty}^{\infty}q^{20n^{2}+6n}+q^{2}\dfrac{E_{10}^{5}}{E_{5}^{2}E_{20}^{2}}
 \sum_{n=-\infty}^{\infty}q^{20n^{2}+14n}\\
 &\quad+2q\dfrac{E_{20}^{2}}{E_{10}}\sum_{n=-\infty}^{\infty}q^{20n^{2}+4n}+2q^{4}\dfrac{E_{20}^{2}}{E_{10}}\sum_{n=-\infty}^{\infty}q^{20n^{2}+16n}\Bigg)\\
 &=\dfrac{E_{10}^{4}f(q^{30},q^{50})f(q,q^{4})}{E_{5}^{5}E_{20}^{2}}
 \Bigg(\sum_{n=-\infty}^{\infty}q^{40n^{2}+2n}+q^{2}\sum_{n=-\infty}^{\infty}q^{40n^{2}+18n}\\
 &\quad-q^{3}\sum_{n=\infty}^{\infty}q^{40n^{2}+22n}-q^{9}\sum_{n=\infty}^{\infty}q^{40n^{2}+38n}\Bigg)\\
 &\quad+\dfrac{E_{10}^{4}f(q^{10},q^{70})f(q,q^{4})}{E_{5}^{5}E_{20}^{2}}
 \Bigg(q^{8}\sum_{n=-\infty}^{\infty}q^{40n^{2}+22n}+q^{14}\sum_{n=-\infty}^{\infty}q^{40n^{2}+38n}\\
 &\quad-q^{7}\sum_{n=\infty}^{\infty}q^{40n^{2}+18n}-q^{5}\sum_{n=\infty}^{\infty}q^{40n^{2}+2n}\Bigg)\\
 &\quad+\dfrac{2E_{20}^{2}E_{80}^{5}f(q,q^{4})}{E_{5}^{3}E_{10}^{2}E_{40}^{2}E_{160}^{2}}
 \Bigg(q\sum_{n=-\infty}^{\infty}q^{40n^{2}+8n}-q^{7}\sum_{n=-\infty}^{\infty}q^{40n^{2}+32n}\Bigg)\\
 &\quad+\dfrac{4E_{20}^{2}E_{160}^{2}f(q,q^{4})}{E_{5}^{3}E_{10}^{2}E_{80}}
 \Bigg(q^{17}\sum_{n=-\infty}^{\infty}q^{40n^{2}+32n}-q^{11}\sum_{n=-\infty}^{\infty}q^{40n^{2}+8n}\Bigg).
\end{align*}

Now, we discuss the following two cases:
\begin{enumerate}[1)]
\item Denote
\begin{align*}
B_{1} &:=f(q,q^{4})\sum_{m=-\infty}^{\infty}q^{20n^{2}+2n}-q^{9}f(q,q^{4})\sum_{n=\infty}^{\infty}q^{40n^{2}+38n},\\
B_{2} &:=q^{2}f(q,q^{4})\sum_{n=-\infty}^{\infty}q^{40n^{2}+18n}-q^{3}f(q,q^{4})\sum_{n=\infty}^{\infty}q^{40n^{2}+22n},\\
B_{3} &:=q^{8}f(q,q^{4})\sum_{n=-\infty}^{\infty}q^{40n^{2}+22n}-q^{7}f(q,q^{4})\sum_{n=\infty}^{\infty}q^{40n^{2}+18n},\\
B_{4} &:=q^{14}f(q,q^{4})\sum_{n=-\infty}^{\infty}q^{40n^{2}+38n}-q^{5}f(q,q^{4})\sum_{n=\infty}^{\infty}q^{40n^{2}+2n},\\
B_{5} &:=qf(q,q^{4})\sum_{n=-\infty}^{\infty}q^{40n^{2}+8n}-q^{7}f(q,q^{4})\sum_{n=-\infty}^{\infty}q^{40n^{2}+32n},\\
B_{6} &:=q^{17}f(q,q^{4})\sum_{n=-\infty}^{\infty}q^{40n^{2}+32n}-q^{11}f(q,q^{4})\sum_{n=-\infty}^{\infty}q^{40n^{2}+8n}.
\end{align*}

By reasoning as above, we obtain
\begin{align*}
H_{5,3}(B_{i})=0,\quad \textrm{for}\quad1\leq i\leq6.
\end{align*}

Therefore,
\begin{align*}
a_{3}(5n+3)=0.
\end{align*}

\item Denote
\begin{align*}
C_{1} &:=f(q,q^{4})\sum_{m=-\infty}^{\infty}q^{20n^{2}+2n}-q^{3}f(q,q^{4})\sum_{n=\infty}^{\infty}q^{40n^{2}+22n},\\
C_{2} &:=q^{2}f(q,q^{4})\sum_{n=-\infty}^{\infty}q^{40n^{2}+18n}-q^{9}f(q,q^{4})\sum_{n=\infty}^{\infty}q^{40n^{2}+38n},\\
C_{3} &:=q^{8}f(q,q^{4})\sum_{n=-\infty}^{\infty}q^{40n^{2}+22n}-q^{5}f(q,q^{4})\sum_{n=\infty}^{\infty}q^{40n^{2}+2n},\\
C_{4} &:=q^{14}f(q,q^{4})\sum_{n=-\infty}^{\infty}q^{40n^{2}+38n}-q^{7}f(q,q^{4})\sum_{n=\infty}^{\infty}q^{40n^{2}+18n},\\
C_{5} &:=qf(q,q^{4})\sum_{n=-\infty}^{\infty}q^{40n^{2}+8n}-q^{7}f(q,q^{4})\sum_{n=-\infty}^{\infty}q^{40n^{2}+32n},\\
C_{6} &:=q^{17}f(q,q^{4})\sum_{n=-\infty}^{\infty}q^{40n^{2}+32n}-q^{11}f(q,q^{4})\sum_{n=-\infty}^{\infty}q^{40n^{2}+8n}.
\end{align*}

Thus, as above,
\begin{align*}
H_{5,4}(C_{i})=0,\quad \textrm{for}\quad1\leq i\leq6.
\end{align*}

Furthermore,
\begin{align*}
a_{3}(5n+4)=0.
\end{align*}
\end{enumerate}
This establishes \eqref{a3}.

The proofs of \eqref{b2} and \eqref{b3} are similar to those of \eqref{a2} and \eqref{a3}.

\section{Final remarks}
On one hand, there are more analogous results on vanishing coefficients in other infinite product expansions beyond this paper. Define
\begin{align}
(-q^{r},-q^{t-r};q^{t})_{\infty}^{3}(q^{s},q^{2t-s};q^{2t})_{\infty} &:=\sum_{n=0}^{\infty}a_{r,s,t}(n)q^{n},\label{a,r-s-t}\\
(-q^{r},-q^{t-r};q^{t})_{\infty}(q^{s},q^{2t-s};q^{2t})_{\infty}^{3} &:=\sum_{n=0}^{\infty}b_{r,s,t}(n)q^{n}\label{b,r-s-t}
\end{align}
where $t\geq5$ is a prime, $r ,s$ are positive integers and $r<t$, $s\neq t$.

Following the same line of proving \eqref{a2}--\eqref{b3}, we can also obtain
\begin{align}
b_{1,4,5}(5n+3) &=b_{2,2,5}(5n+4)=0,\label{a14-22}\\
a_{1,1,7}(7n+2) &=a_{1,1,7}(7n+5)=0,\\
b_{1,5,7}(7n+3) &=b_{1,5,7}(7n+4)=0,\\
a_{3,3,7}(7n+5) &=a_{3,3,7}(7n+6)=0,\\
b_{3,1,7}(7n+3) &=b_{3,1,7}(7n+5)=0,\\
a_{4,6,11}(11n+3) &=b_{4,2,11}(11n+3)=0,\\
a_{5,2,11}(11n+1) &=b_{5,8,11}(11n+4)=0.\label{a52-58}
\end{align}

There are other equalities similar to \eqref{a14-22}--\eqref{a52-58} for $t=7$ or 11. However, there are no similar results for $t=13$ or 17. It is natural to ask whether or not there exists a criterion which can be used for searching for vanishing coefficients of the arithmetic progressions in $a_{r,s,t}(n)$ and $b_{r,s,t}(n)$.

On the other hand, the product representation of the Rogers-Ramanujan fractions is given by \cite[Eq. (16.2.1)]{Hirb2017}:
\begin{align}
R(q) =\dfrac{(q,q^{4};q^{5})_{\infty}}{(q^{2},q^{3};q^{5})_{\infty}}=\sum_{n=0}^{\infty}u(n)q^{n}.\label{RR-fraction}
\end{align}

Richmond and Szekeres \cite{RS1978} also examined asymptotically the power series coefficients of a large class of infinite products including \eqref{RR-fraction}. They proved that, for sufficiently large $n$,
\begin{align}\label{sign pattern}
u(5n), u(5n+2)>0,\quad \textrm{and} \quad u(5n+1), u(5n+3), u(5n+4)<0.
\end{align}
A similar result was also obtained for the coefficients of $1/R(q)$.

In 1981, Andrews \cite{And1981} used partition-theoretic interpretations of these coefficients, to prove \eqref{sign pattern} holds for all $n\geq0$, except that $u(3)=u(8)=u(13)=u(23)=0$. Hirschhorn \cite{Hir1998} later provided an elementary proof of \eqref{sign pattern} using only the quintuple product identity.

With the aid of computer, the signs of coefficients in $q$-series \eqref{a,r-s-t} and \eqref{b,r-s-t} appear to be periodic from some large $n$ for $t=5, 7$, and 11. For example,
\begin{align*}
b_{1}(5n), b_{1}(5n+2), b_{1}(5n+3)>0, \quad&\textrm{and}\quad b_{1}(5n+4)<0,\\
a(10n), a(10n+3), a(10+6)>0, \quad&\textrm{and}\quad a(10n+1), a(10n+5), a(10n+8)<0,\\
a_{1,1,7}(7n+1), a_{1,1,7}(7n+6)>0, \quad&\textrm{and}\quad a_{1,1,7}(7n+3), a_{1,1,7}(7n+4)<0,\\
b_{3,3,11}(22n+2)>0, \quad&\textrm{and}\quad b_{3,3,11}(22n+13)<0.
\end{align*}
However, it is unclear how these inequalities could be proved by $q$-series.

Furthermore, the numerical evidence suggests the following conjecture.
\begin{conjecture}
For given $r, s$, and $t$, the signs of $a_{r,s,t}(n)$'s and $b_{r,s,t}(n)$'s are periodic with period $t$ or $2t$ from sufficiently large $n$.
\end{conjecture}

\section*{Acknowledgement}
The author is indebted to Shishuo Fu for his helpful comments on a preliminary version of this paper. The author would like to acknowledge the referee for his/her careful reading and helpful comments on an earlier version of the paper. This work was supported by the National Natural Science Foundation of China (No.~11501061) and the Fundamental Research Funds for the Central Universities (No. 2018CDXYST0024).

\end{document}